\documentclass[conference,letterpaper]{IEEEtran}

\usepackage[utf8]{inputenc} 
\usepackage[T1]{fontenc}
\usepackage{url}
\usepackage{ifthen}
\usepackage{cite}
\usepackage{color}
\usepackage[cmex10]{amsmath} 
\usepackage{verbatim,amssymb,amsfonts,amscd, algorithmic,amsthm,hyperref}

\usepackage{graphicx}
\usepackage{mathtools}
\DeclareMathOperator*{\argmax}{arg\,max}

\DeclareMathOperator*{\esssup}{ess\,sup}
\DeclareMathOperator*{\ER}{\mathrm{ER}}
\newcommand{\tc}{T_{\scriptscriptstyle \text{C}}}
\newcommand{\twg}{T_{\scriptscriptstyle \text{G}}}

\newtheorem{theorem}{Theorem}
\newtheorem{lemma}[theorem]{Lemma}

\newtheorem{remark}[theorem]{Remark}

\begin{document}
\title{Optimality of Graph Scanning Statistic for Online Community Detection} 


\author{%
  \IEEEauthorblockN{Liyan Xie and Yao Xie}
  \IEEEauthorblockA{H. Milton Stewart School of Industrial and Systems Engineering\\
                    Georgia Institute of Technology\\
                    Atlanta, GA 30332, United States\\
                    Email: lxie49@gatech.edu, yao.xie@isye.gatech.edu}
}

\maketitle

\begin{abstract}
Sequential change-point detection for graphs is a fundamental problem for streaming network data types and has wide applications in social networks and power systems. Given fixed vertices and a sequence of random graphs, the objective is to detect the change-point where the underlying distribution of the random graph changes. In particular, we focus on the local change that only affects a subgraph. We adopt the classical Erd\H{o}s-R\'enyi model and revisit the generalized likelihood ratio (GLR) detection procedure. The scan statistic is computed by sequentially estimating the most-likely subgraph where the change happens. We provide theoretical analysis for the asymptotic optimality of the proposed procedure based on the GLR framework. We demonstrate the efficiency of our detection algorithm using simulations. 
\end{abstract}

\section{Introduction}

Change-point detection is a fundamental problem for network data, such as power systems \cite{rovatsos2017statistical,chen2015quickest}, sensor networks \cite{xie2020sequential}, and social networks \cite{wang2017fast,peel2015detecting,li2017detecting}. Network data can be modeled as graphs. For instance, in social networks, each node represents users, and the edge represents the connectivity between users. We consider the Erd\H{o}s-R\'enyi model \cite{erdHos1960evolution}, which is parameterized by the probability of having an edge between two nodes. 
In this paper, we consider the detection of a {\it local} change in the graph, which only affects the distribution of a {\it subgraph}. The detection procedure is to form the scan statistic based on the Erd\H{o}s-R\'enyi model. In particular, we treat the affected subgraph as unknown anomaly information, and apply the generalized likelihood ratio test \cite{lai-ieeetit-1998} to form the scan statistic, utilizing the graph scanning techniques \cite{priebe2005scan,sharpnack2016detecting}. 

In order to give a computationally efficient algorithm and theoretical calibration, we assume the size of the subgraph affected by the change is {\it known}. More specifically, when the change happens, only a subset of the graph, of known size, is affected by the change and has a different distribution, while the distribution for the rest of the graph remains the same. The problem of local change-point detection is challenging for two reasons: (i) it is not clear whether there is a change; (ii) if there is a change at some time, it is not clear which subgraph contains the change. 

The major motivating application of our study is community detection. In particular, we are interested in the detection of the {\it emergence} of a community in a network that is homogeneous in the beginning. Such a problem is essential for dynamic networks. For example, the social network can start from a homogeneous state, and then evolve over time and form a community. Usually, the interactions between nodes within the community are more dense than other parts of the network. Another example is ambient noise monitoring in seismic sensor networks. More specifically, the cross-correlation function between the sensors affected by the change will have a significant peak at the time of the change. Meanwhile, such waveform does not exist for cross-correlation functions between affected sensors and unaffected sensors, and among unaffected sensors. Therefore, this problem, mathematically, becomes detecting a local change in a sequence of graphs \cite{he2018sequential}. 

In this paper, we focus on the parametric approach for constructing {\it scan statistics} to detect a local change in a sequence of graphs. For simplicity, we adopt the Erd\H{o}s-R\'enyi model (while the analysis can be generalized to more complicated models), where each edge exists with probability $p_0\in(0,1)$ and independently with each other. After the change, the affected subgraph still follows the Erd\H{o}s-R\'enyi model, but with a different parameter $p_1\in(0,1)$. We consider a {\it sequential} detection setting and prove the optimality of the online detection procedure in the sense that the detection delay matches the well-known lower bound. The main idea of the proof is adopted from the seminal work on generalized likelihood ratio test \cite{lai-ieeetit-1998}.

This paper is related to works in community detection and graph scan statistics. 
In \cite{marangoni2015sequential}, three likelihood ratio based algorithms were developed for detecting communities in the Erd\H{o}s-R\'enyi graph, including the exhaustive search, the mixture, and the hierarchical mixture methods. Theoretical approximation was also given in \cite{marangoni2015sequential} to characterize the false alarms. 
In \cite{arias2014community}, the community detection for Erd\H{o}s-R\'enyi graphs was considered as a hypothesis testing problem, where the goal is to find a test function that takes the random graph as input and claim whether there is a community or not. The detectability of this problem in the asymptotic dense regime, when the connection probability $p_0$ is large enough, was provided in \cite{arias2014community}. Later on, information-theoretic lower bounds for the asymptotically sparse regime, when the connection probability $p_0$ is small enough, was studied in \cite{verzelen2015community}.

The rest of the paper is organized as follows. We present the problem setup in Section~\ref{sec:setup}. The detection procedure is detailed in Section~\ref{sec:method}. We present the optimality study in Section~\ref{sec:optimality}. Numerical examples are presented in Section~\ref{sec:numerical} to support the theoretical findings. Finally Section~\ref{sec:conclusion} contains our concluding remarks.

\section{Problem setup}\label{sec:setup}

Given a network with $N$ sensors (nodes) numbered as $\{1,\ldots, N\}$, the dynamic graphical structure is observed as a sequence of {\it undirected} adjacency matrixes $G^{(1)},G^{(2)},\ldots,$ where $G^{(t)} \in \{0,1\}^{N\times N}$ characterizes the edge or interaction information between different nodes, i.e., $G_{ij}^{(t)} =1$ if and only if there is an edge between node $i$ and node $j$ at time $t$.
Consider the Erd\H{o}s-R\'enyi model, denoted as $\ER(N,p)$, where the graph is constructed by connecting nodes randomly. Each edge is included with probability $p$ independently, i.e., $\mathbb P(G_{ij}^{(t)} = 1) = p$. The sequence of observations $G^{(1)},G^{(2)},\ldots$ are independent realizations of the Erd\H{o}s-R\'enyi model.

Assume that there is a change-point at an {\it unknown} time $\tau$ that changes the distribution of a subgraph with nodes indexed by $V^{*} \subset \{1,\ldots,N\}$. Before the change, the full graph follows the model $\ER(N,p_0)$ with connection probability $p_0$. After the change, the subgraph $V^*$ follows $\ER(|V^*|,p_1)$, i.e., the connection probability inside the subgraph $V^*$ becomes $p_1$, with everything else the same. Here $|V^*|$ denotes the cardinality of the set $V^*$. Usually, the true subgraph $V^*$ is unknown as it represents the anomaly information. We assume that the cardinality of $V^*$ equals a known constant $n$ with $n <N$. In most applications, we have $n\ll N$ which means that we are only interested in the local graphical changes.

Although the true subgraph where the change happens is unknown, there are only a {\it finite} number of possible subgraphs when the total number of nodes $N$ is fixed. Denote all possible subgraphs as: 
\begin{equation*}\label{eq:subgraph}
  \mathcal V = \{V^{(1)}, \dots, V^{(d)}\}.  
\end{equation*}
Note that the number of all possible subgraphs can be upper bounded by ${N \choose n}$ and we further have 
\begin{equation*}\label{eq:constr_d}
d \ll {N \choose n},
\end{equation*}
if we are only interested in locally {\it connected} subgraphs, which is a reasonable assumption for community detection tasks where the change tends to happen within a small neighborhood. 

In summary, the problem of detecting a local change for the underlying Erd\H{o}s-R\'enyi model becomes the following hypothesis testing problem:
\begin{equation}
\label{eqn: hypo}
\begin{array}{ll}
H_0: & \mathbb P (G_{ij}^{(t)} = 1) = p_0,\  \forall i,j; \ t=1,2,\ldots    \\
H_{1}: & \mathbb P (G_{ij}^{(t)} = 1) = p_0,\  \forall i,j; \ t = 1,2,\ldots, \tau-1   \\
&   \mathbb P (G_{ij}^{(t)} = 1) = p_1, \ \forall i,j \in V^*; \ t = \tau,\tau+1,\ldots  \\
&  \mathbb P (G_{ij}^{(t)} = 1) = p_0, \ \forall i \text{ or } j  \notin V^*; \ t = \tau,\tau+1,\ldots 
\end{array}
\end{equation}
where $\tau$ represents the change-point. This hypothesis testing problem is illustrated in Fig.~\ref{Fig: change-point}, where the post-change subgraph contains only three nodes and is shown in highlight. 

\begin{figure}[ht!]
\begin{center}
    \includegraphics[width=1\linewidth]{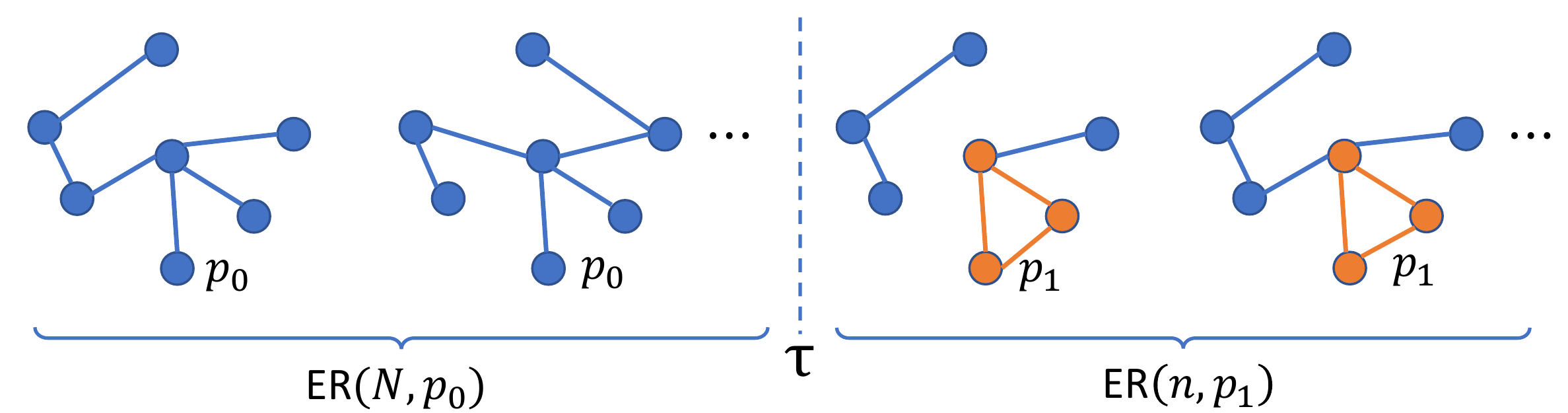}
\end{center}
\caption{Graphs prior to the change-point in time $\tau$ follow the Erd\H{o}s-R\'enyi model with coneection probability $p_0$. After the change-point $\tau$, the subgraph (shown in highlight) follows the Erd\H{o}s-R\'enyi model with connction probability $p_1\neq p_0$, with everything else the same. We are particularly interested in detecting the local change in the subgraph.}
\label{Fig: change-point}
\end{figure}
Given access to a sequence of graph observations $G^{(1)},G^{(2)},\ldots$, if they are sampled from the hypothesis $H_1$ in the model \eqref{eqn: hypo}, our objective is to design a stopping time that can detect the unknown change-point $\tau$ as quickly as possible. Meanwhile, if the data are sampled from the hypothesis $H_0$ in the model \eqref{eqn: hypo}, it is desired to have less false alarms as possible. Here we assume the connection probabilities $p_0, p_1$ are known, but the subgraph where the change happens is {\it unknown}.

\section{Detection Procedure}\label{sec:method}

The detection problem \eqref{eqn: hypo} can be solved based on statistical change detection methodology, which we describe in this section. We start by introducing the basic cumulative sum (CUSUM) procedure for change detection when the subgraph $V^*$ is known, and then study the generalized likelihood ratio (GLR) test for unknown subgraphs. 

The log-likelihood ratio between the pre- and post-change distributions plays a key role in sequential change detection. For the Erd\H{o}s-R\'enyi model $\ER(N,p_0)$ before change, we have the likelihood function of $G^{(t)}$ is
\[
L(G^{(t)},p_0) =  \prod_{\substack{1\leq i<j\leq N }} p_0^{G_{ij}^{(t)}} (1-p_0)^{1- G_{ij}^{(t)}},
\]
and the likelihood function for post-change model is similarly defined. Therefore, the log-likelihood ratio for the hypothesis testing problem \eqref{eqn: hypo} with subgraph $V^*$ reads
\[  
\ell_{V^*}(G^{(t)}) = \sum\limits_{\substack{i,j \in V^* \\ i<j}} \left[G_{ij}^{(t)} \log  \frac{p_1}{p_0} + (1-G_{ij}^{(t)}) \log\frac{1-p_1}{1-p_0}\right].
\]

When the subgraph $V^*$ is known, we can construct the well-known CUSUM procedure, which was first introduced by Page in \cite{page-biometrica-1954}. The CUSUM statistic can be formed recursively as
\begin{equation}\label{eq:cusum}
S_{t} =  (S_{t-1})^{+} + \ell_{V^*}(G^{(t)}), \ S_{0} = 0,
\end{equation}
where $(x)^{+}:=\max\{x,0\}$. 
The corresponding stopping time is: 
\begin{equation}\label{eq:cusum_stop}
    \tc = \inf\{ t: S_{t} > b \},
\end{equation}
where the threshold $b$ is a pre-set constant to control the false alarm rates. 

By Jensen's inequality, it is easy to show that the expectation of the increment term $\ell_{V^*}(G^{(t)})$ in \eqref{eq:cusum} is negative under the pre-change regime, and positive in the post-change regime. Therefore, the CUSUM statistic $S_t$ will have a positive drift after the change happens, enabling its efficient detection of the change-point. The  CUSUM procedure was shown to have strong optimality properties in
\cite{Lorden1971,mous-astat-1986,ritov-astat-1990,lai-ieeetit-1998}. In particular, it attains the minimal worst-case detection delay among all testing procedures that satisfy certain false alarm constraint.

However, the CUSUM statistic \eqref{eq:cusum} cannot be used directly when the changed subgraph $V^*$ is {\it unknown}. Therefore, we adopt the GLR framework \cite{lai-ieeetit-1998,marangoni2015sequential}, also known as the scan test \cite{arias2014community,he2018sequential}. The GLR test was originally developed for change detection for parametric families when the post-change parameter is unknown, by substituting the maximum likelihood type estimators. Here instead of estimating the post-change parameters, we estimate the unknown post-change subgraph by maximizing the likelihood function of samples in the past. More specifically, the GLR statistic at time $t$ is defined as
\[
\mathcal S_t = \max_{1\leq k\leq t} \max_{V \in \mathcal V} R_{t,k,V},
\]
where $R_{t,k,V}$ is the log-likelihood ratio of samples $G^{(k)},\ldots,G^{(t)}$, assuming the changed subgraph is $V$ and the change-point $\tau=k$, which can be derived as: 
\begin{equation}\label{eq:ratio}
R_{t,k,V} = \sum\limits_{\substack{i,j \in V \\ i<j}} \sum\limits_{m=k}^t \left[G_{ij}^{(m)} \log  \frac{p_1}{p_0} + (1-G_{ij}^{(m)}) \log\frac{1-p_1}{1-p_0}\right].
\end{equation}

Usually, the GLR statistics constructed this way does not enjoy recursive computations as in \eqref{eq:cusum}; the computation complexity for maximizing over all possible $k\leq t$ will increase polynomially with time $t$. To save the computational cost, we use the window-limited GLR approach \cite{lai-ieeetit-1998}, which we detailed as follows. At each time $t$, we form statistics by aggregating test statistic under different subgraphs $V$, over time 
\[
(t-m_\alpha,t-m'_\alpha),
\] 
where $m'_\alpha$ is the minimal number of samples to ensure the reliable estimation of the subgraph, and $m_\alpha$ is the maximum number of samples to look back in order to reduce the computation time. 
The window-limited GLR stopping time is \cite{marangoni2015sequential}
\begin{equation}\label{eq:GLR_stat}
\twg = \inf\left\{ t: \max_{ t-m_\alpha \leq k \leq t-m'_\alpha} \max_{V\in\mathcal V} R_{t,k,V} > b \right\}, 
\end{equation}
where $b$ is the threshold, and $R_{t,k,V}$ is the log-likelihood ratio as defined in \eqref{eq:ratio}. 


We are further interested in knowing which subgraph contains the change in the graph structure. Once we have detected the change-point as $k$, we can choose a post-change interval $(k,t)$. The test statistic $R_{t,k, V}$ is useful in localizing the change, as the subgraph $V^{*}$ that maximizes $R_{t,k, V}$, over all possible subgraphs in $\mathcal V$, is the {\it maximum likelihood estimate} (MLE) of the subgraph containing the change,
\begin{equation}\label{eq:densestn}
\widehat V_{k,t} = \argmax_{V\in \mathcal{V}} R_{t,k, V}.
\end{equation}
It is worth mentioning that the solution to \eqref{eq:densestn} is the so-called ``densest $n$ subgraph'' \cite{feige2001dense} and it is an NP-hard problem; there is no constant approximation ratio algorithm due to its core hardness. We use the greedy procedure in \cite{feige2001dense} to approximate the maximum likelihood estimate $\widehat V_{k,t}$. 
For completeness, we restate the procedure here: ``Sort the vertices by order of their degree. Let $H$ denote the $n/2$ vertices with highest degrees in the graph $G$. Sort the remaining vertices by the number of neighbors they have in $H$. Let $C$ denote the $n/2$ vertices in $G \setminus H$ with the largest number of neighbors in $H$. Return $H \cup C$.''

\section{Optimality}\label{sec:optimality}

We first introduce two metrics commonly used to characterize the performance of detection procedures in sequential change detection. 

The average run length (ARL) is defined as the average time between false alarms when there is no change; it can be denoted as $\mathbb{E}_\infty[T]$, where $\mathbb{E}_\infty$ is the expectation under the pre-change measure (i.e., the change-point is at $\infty$). 
The expected detection delay (EDD) refers to the expected delay in detecting the change. 
There are two common definitions for EDD as introduced in \cite{Lorden1971} and \cite{poll-astat-1985}. We adopt the one in \cite{Lorden1971} as follows:
\begin{equation}\label{eq:edd}
\bar{\mathbb{E}}_0(T) = \sup_{\tau \geq 1} \esssup \mathbb{E}_\tau[(T-\tau)^{+}|G^{(1)},\ldots,G^{(\tau-1)}],    
\end{equation}
where the essential supremum is taken over all possible change-point $\tau$ and realizations $G^{(1)},\ldots,G^{(\tau-1)}$ before the change; and $\mathbb{E}_\tau$ is the expectation under the probability measure that the change-point equals to $\tau$. 

The goal is to minimize EDD subject to the ARL constraint that $\mathbb{E}_\infty[T]\geq\gamma$ for a positive constant $\gamma$. 

The lower bound to the worst-case EDD $\bar{\mathbb{E}}_0(T)$ was given in  \cite[Theorem~1]{lai-ieeetit-1998}. 
More specifically, we restate this lower bound in our setting \eqref{eqn: hypo}.
\begin{theorem}[\cite{lai-ieeetit-1998}]\label{thm:optimalEDD}
As $\gamma \rightarrow \infty$, we have 
\[
\inf \left\{ \bar{\mathbb{E}}_0(T): \mathbb{E}_\infty(T) \geq \gamma  \right\} \geq (I^{-1} + o(1))\log\gamma,
\]
where 
\[
I = {n \choose 2}\left(p_1\log \frac{p_1}{p_0} + (1-p_1)\log \frac{1-p_1}{1-p_0}\right),
\] 
is the Kullback–Leibler (KL) divergence between the graphical distribution $\ER(n,p_1)$ and $\ER(n,p_0)$ on the changed subgraph $V^*$ with $|V^*|=n$.
\end{theorem}
Theorem~\ref{thm:optimalEDD} means that for any detection procedure with ARL greater than $\gamma$, the minimal detection delay is of the order of $\log\gamma/I$. Therefore, a detection procedure is called {\it first-order asymptotic optimal} if its EDD equals to  $\log\gamma/I(1+o(1))$ asymptotically as $\gamma\rightarrow\infty$. Below, we prove that the detection rule \eqref{eq:GLR_stat} achieves the lower bound in Theorem~\ref{thm:optimalEDD}. 

We first consider the pre-change regime and state the following lemma:
\begin{lemma}\label{thm:lower_bound}
For the stopping time \eqref{eq:GLR_stat}, we have
\begin{equation}\label{eq:arl_constr}
\sup_{\tau \geq 1} \mathbb{P}_\infty( \tau \leq \twg < \tau + m_\alpha) \leq 2m_\alpha  e^{-b} {N \choose n}.     
\end{equation}
\end{lemma}
\begin{proof}
First note that
\begin{equation}\label{eq:conditional} 
\begin{split}
& \mathbb{P}_\infty( \tau \leq \twg \leq \tau + m_\alpha ) \\
\leq & 
\sum_{\tau - m_\alpha \leq k \leq \tau + m_\alpha} \mathbb{P}_\infty ( \tau_k \leq k + m_\alpha), 
\end{split}
\end{equation}
where 
\begin{equation}\label{eq:tauk}
\tau_k := \inf\left\{ t \geq k + m'_\alpha: \widehat{V}_{k,t} \in \mathcal V, \text{ and } R_{t,k,\widehat{V}_{k,t}} \geq b \right\}. 
\end{equation}
To analyze $\mathbb{P}_\infty(\tau_k \leq k + m_\alpha)$, we use a {\it change-of-measure} argument. Let $\mathbb{P}_V^k$ denote the probability measure under which the distribution of $G^{(i)}$ is $\ER(N,p_0)$ for $i < k$, and the distribution of the subgraph $V$ becomes $\ER(n,p_1)$ for $i \geq k$. 
Define a measure 
\[
Q^{k} =\sum_{V \in \mathcal V} \mathbb{P}_V^k.
\] 
Since $\mathcal V$ is a finite set, $Q^{k}$ is a finite measure. 
For $t \geq k$, let $\mathcal F_{k,t}$ denote the sigma-algebra generated by $G^{(k)},\ldots,G^{(t)}$. The Radon-Nikodym derivative of the restriction of measure $Q^{k}$ to $\mathcal F_{k,t}$ relative to the restriction of $\mathbb{P}_\infty$ to $\mathcal F_{k,t}$ is 
\[
L_t =  \sum_{V \in \mathcal V} \exp\left\{R_{t,k,V} \right\}.
\]
Hence by Wald's likelihood ratio identity, we have
\[
\begin{aligned}
\mathbb{P}_\infty(\tau_k \leq k + m_\alpha) & = \int_{\{\tau_k \leq k + m_\alpha\}} L^{-1}_{\tau_k} dQ^k
 \\
 & = \sum_{V \in \mathcal V} \int_{\{\tau_k \leq k + m_\alpha\}} L^{-1}_{\tau_k} dP_V^k.
 \end{aligned}
\]
Since
\begin{equation}\label{eq:ineq}
L_{\tau_k} = \sum_{V \in \mathcal V} \exp\left\{R_{\tau_k,k,V} \right\} \geq \exp\left\{R_{\tau_k,k,\widehat{V}_{k,t}}\right\}  \geq e^b,
\end{equation}
where the last inequality is due to the definition of $\tau_k$ in \eqref{eq:tauk}, we have that
\[
\begin{aligned}
\sup_{k} \mathbb{P}_\infty ( \tau_k \leq k + m_\alpha) & = \sup_{k} \sum_{V \in \mathcal V} \int_{\{\tau_k \leq k + m_\alpha\}} L^{-1}_{\tau_k} dP_V^k \\
& \leq  e^{-b}|\mathcal V| \leq e^{-b} {N \choose n}.
\end{aligned}
\]
Substitute into \eqref{eq:conditional}, we have
\[ 
\mathbb{P}_\infty( \tau \leq \twg \leq \tau + m_\alpha )  \leq 
2m_\alpha  e^{-b} {N \choose n}.
\]
\end{proof}


Given the condition \eqref{eq:arl_constr} in Lemma~\ref{thm:lower_bound}, by \cite[Theorem~4]{lai-ieeetit-1998}, we have the following optimality results for our setup \eqref{eqn: hypo}.
\begin{theorem}\label{thm:optim}
For the stopping time \eqref{eq:GLR_stat}, if the window size $m_\alpha$ satisfies
\[
\lim\inf_{\alpha\rightarrow 0} \frac{m_\alpha }{\left| \log\alpha \right|} > I^{-1},
\quad \log m_\alpha = o(\log\alpha),
\]
and the threshold $b$ satisfies
\begin{equation}\label{eq:b}
 2m_\alpha  e^{-b} {N \choose n} = \alpha,
\end{equation}
then we have 
\begin{equation}\label{eq:ARL}
\mathbb{E}_\infty[\twg]\geq (\frac12-\alpha)(\frac{m_\alpha}{2\alpha}-1),
\end{equation}
and as $\alpha \rightarrow 0$,
\[
\bar{\mathbb{E}}_0(\twg) \leq (I^{-1} + o(1)) b, \quad \text{as} \; b \sim |\log\alpha| \rightarrow \infty.
\]
Therefore the stopping rule \eqref{eq:GLR_stat} is asymptotically optimal.
\end{theorem}
\begin{proof}
For the true subgraph $V^*$, define the window-limited CUSUM rule as
\[
\widetilde T = \inf\left\{ t: \max_{ t-m_\alpha \leq k \leq t-m'_\alpha} R_{t,k,V^*} > b \right\}.
\]
It is obvious that 
\[
\bar{\mathbb{E}}_0(\twg) \leq \bar{\mathbb{E}}_0(\widetilde T)
\]
due to the definition of $\twg$ in \eqref{eq:GLR_stat}. 
Note that for any finite value $N,n$, the constraint \eqref{eq:b} 
implies that we can choose the threshold as $b\sim|\log\alpha|$.
By \cite[Theorem~4]{lai-ieeetit-1998}, we have as $b \sim |\log\alpha| \rightarrow \infty$,
\[
\bar{\mathbb{E}}_0(\widetilde T) \leq  (I^{-1} + o(1)) b,
\] 
if $m'_\alpha = o(| \log \alpha|)$. The ARL \eqref{eq:ARL} is proved in \cite{lai-ieeetit-1998} whenever $\twg$ satisfies the condition \eqref{eq:arl_constr}. Further note that 
\[
\log\mathbb{E}_\infty[\twg] \sim | \log \alpha|.
\]  
Therefore, $\twg$ matches the lower bounds in Theorem~\ref{thm:optimalEDD}. In other words, $\twg$ is first-order asymptotically optimal.
\end{proof}

It is worth mentioning that the condition \eqref{eq:b} only yields $b \sim |\log\alpha|$ for finite values of $N,n$. It does not hold for $N$ that diverges to infinity. 
The analysis here differs from the original GLR framework considered in \cite{lai-ieeetit-1998} and can be viewed as a special case, since the unknown subgraph has finite possibilities while the unknown parameter in parametric models can vary in a continuous space. 

\begin{remark}[Generalization to unknown $p_1$]
When the post-change probability $p_1$ is unknown, we can estimate it using the maximum likelihood estimator when formulating the GLR detection statistic. More specifically, given $G^{(k)},\ldots,G^{(t)}$ and a subgraph $V$, the MLE of $p_1$ is given by
\[
\widehat p_1^{(k,t,V)} = \sum\limits_{i,j\in V,i<j} G_{ij}^{(m)}/{n \choose 2}.
\]
And the resulted detection procedure is
\begin{equation}\label{eq:GLR_stat2}
\twg' = \inf\left\{ t: \max_{ t-m_\alpha \leq k \leq t-m'_\alpha} \max_{V\in\mathcal V} U_{t,k,V} > b \right\},
\end{equation}
where $U_{t,k,V}$ is defined as
\[  
\sum\limits_{\substack{i,j\in V\\ i<j}} \sum\limits_{m=k}^t \left[G_{ij}^{(m)} \log  \frac{\widehat p_1^{(k,t,V)}}{p_0} + (1-G_{ij}^{(m)}) \log\frac{1-\widehat p_1^{(k,t,V)}}{1-p_0}\right].
\]
\end{remark}

\begin{remark}[Generalization to directed graph]
In multivariate time series models or point processes, the underlying connectivity can be modeled as a {\it directed} graph, in contrary to the undirected graph considered in this paper. It is worth mentioning that the aforementioned maximum likelihood estimate and the optimality results still holds for sequential change detection on directed graphical models, since the log-likelihood ratio can be computed in a similar fashion. 
\end{remark}

\section{Numerical Examples}\label{sec:numerical}

We present simulation examples using the Erd\H{o}s-R\'enyi graphical models to visualize the detection procedure \eqref{eq:GLR_stat}. 

The size of the network $N$, i.e., the total number of nodes, is set as $20$ and $50$, respectively. In both cases, we are interested in detecting the change that happens only in a much smaller subgraph consisting of $n=5$ nodes. The pre-change edge-forming probability is set as $p_0=0.2$, and the post-change probability is $p_1=0.5$, i.e., the change increases the intensity of edges within the changed subgraph. 

For numerical issues, we do not compute the worst-case detection delay \eqref{eq:edd} that takes supremum over all possible past observations and over all possible change-points. Instead, we compute an alternative formulation $\mathbb E_1[\twg]$ that can be conveniently evaluated by setting the change-point as $\tau=1$, i.e., the change happens before we take any sample. 

In Fig.~\ref{Fig:EDD_WLGLR}, we compare the EDD of the GLR procedure in \eqref{eq:GLR_stat} and the CUSUM procedure in \eqref{eq:cusum_stop}. The CUSUM statistic serves as a baseline since it is the optimal detection procedure with the smallest detection delay. It is shown that the detection delay of the GLR approach indeed matches the detection delay of CUSUM in first-order (i.e., in the slope). Moreover, it can be seen that the detection delay of the GLR approach tends to increase as we increase the network size $N$, since it becomes more difficult to scan for the right subgraph containing the change.

\begin{figure}[h!]
\begin{center}
\begin{tabular}{cc}
\includegraphics[width = 0.46\linewidth]{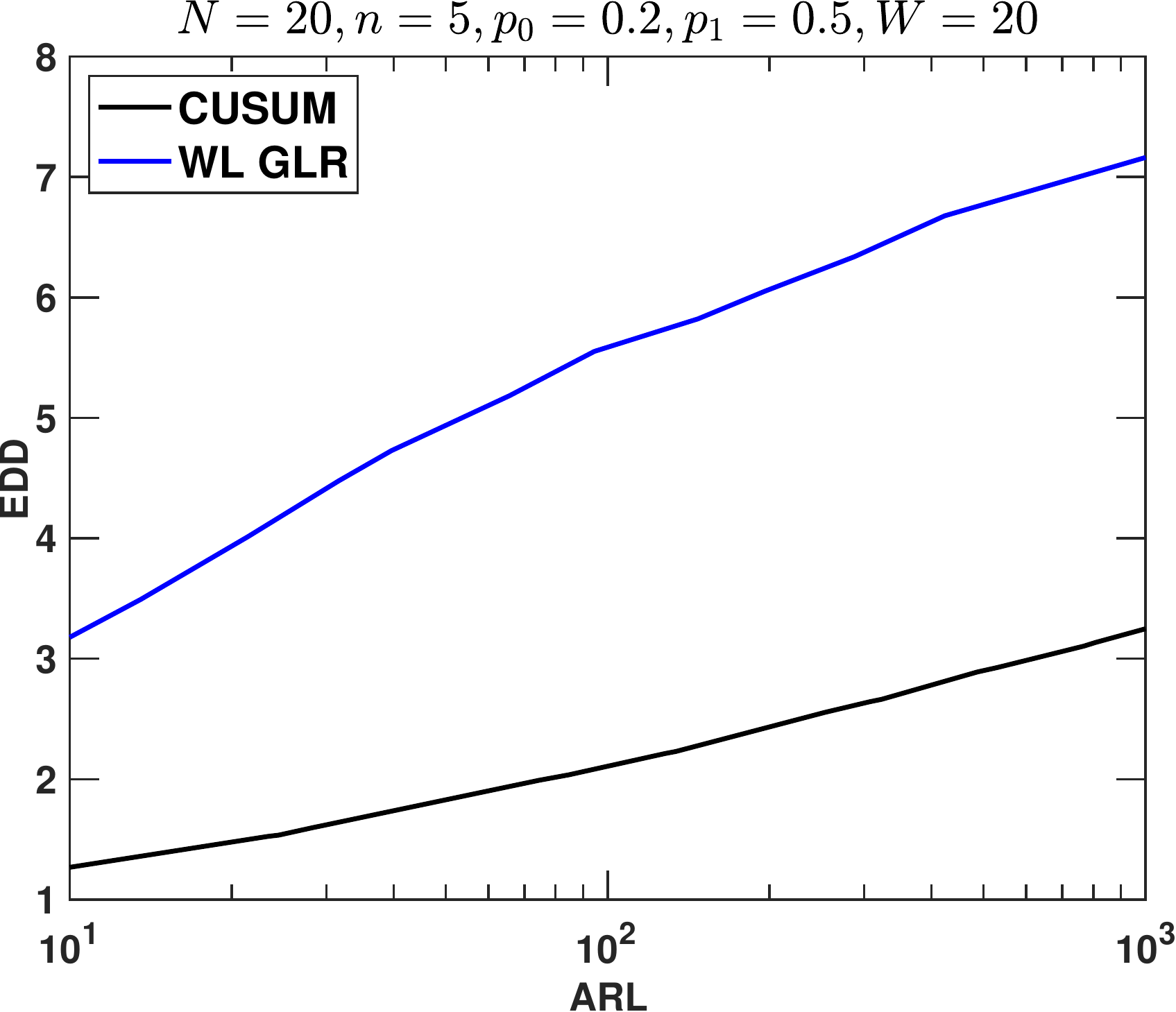} & \includegraphics[width = 0.46\linewidth]{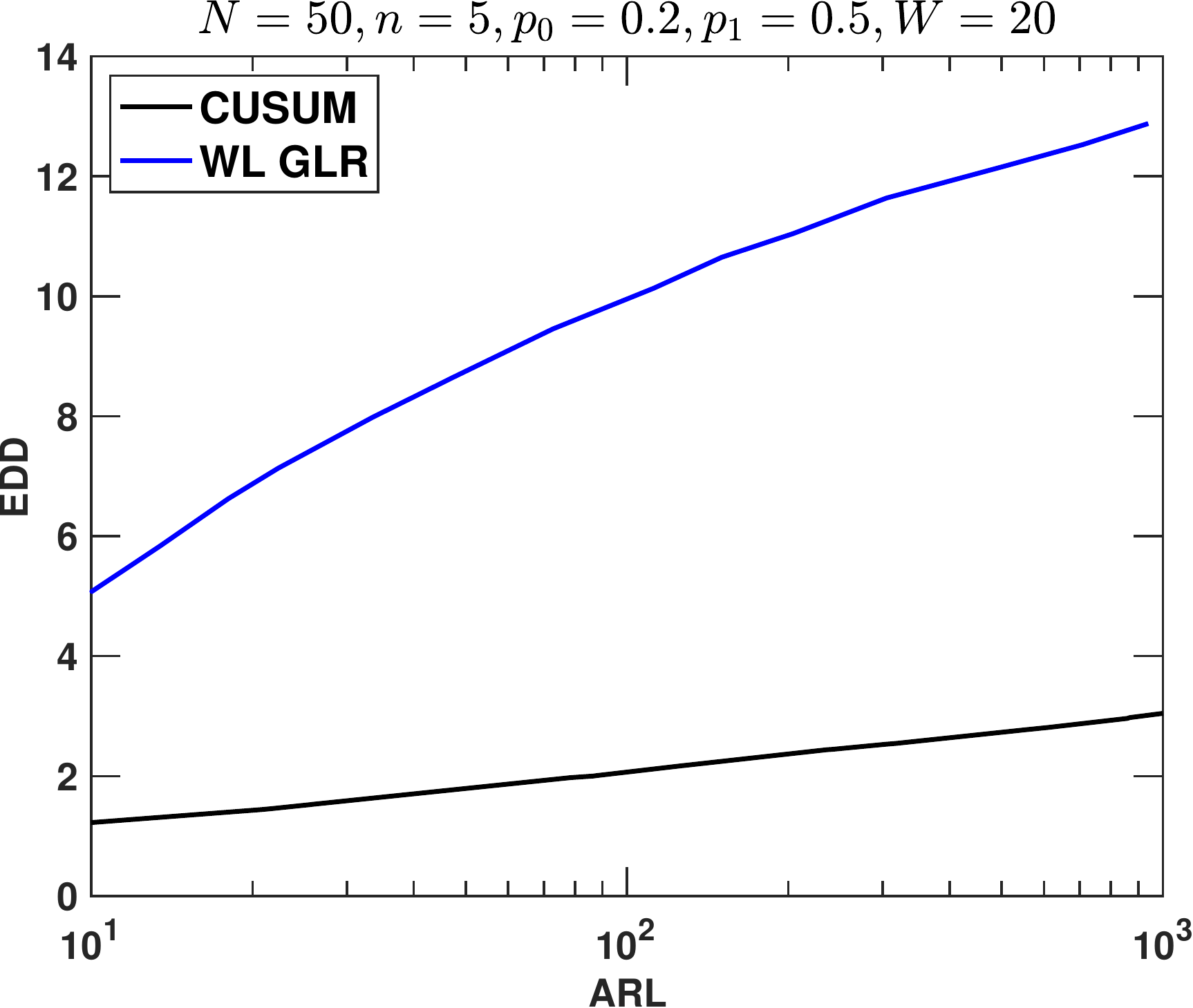}
\end{tabular}
\end{center}
\caption{The EDD/ARL tradeoff for window-limited GLR by graph scanning, and the optimal CUSUM when the subgraph is known. Left: $N=20$; Right: $N=50$.}
\label{Fig:EDD_WLGLR}
\end{figure}

\section{Conclusion}\label{sec:conclusion}
We have revisited the sequential change detection for Erd\H{o}s-R\'enyi graphs. The problem setup can be applied to community detection problems. 
The graph scanning statistic considered in this paper is formed by scanning all possible subgraphs over the whole graph. The detection procedure matches the well-known GLR test and is asymptotically optimal. 

Future direction includes extending the proposed method to more complicated graphical models and to sequences with dependency.
Moreover, the framework can be applied to the {\it goodness-of-fit test for local regions}. The global null is that a known graphical distribution $P_0$ (e.g., $\ER(N,p_0)$) is a good fit for all local regions (the whole graph). The alternative is that there is a subgraph such that the underlying distribution distinct from $P_0$ significantly. For each local region, we can compute a local test statistic based on GLR, and compare it with a threshold, which can be set by simulation or the limiting distribution of the test statistic. 

\section*{Acknowledgment}
The work of Liyan Xie and Yao Xie is partially supported by an NSF CAREER Award CCF-1650913, DMS-1938106, DMS-1830210, CCF-1442635, and CMMI-1917624.

\bibliographystyle{IEEEtran}
\bibliography{scan_stat}

\end{document}